\documentclass[12pt,draftcls,journal,onecolumn]{IEEEtran}

\usepackage{amssymb,amsthm, amsmath,latexsym}
\usepackage{graphicx}
\usepackage{mathrsfs}
\usepackage{amsfonts}
\usepackage{amssymb}
\usepackage{longtable}
\usepackage{amsmath}
\usepackage{setspace}
\usepackage{caption}
\usepackage[figuresright]{rotating}
\usepackage[misc]{ifsym}
\usepackage{bbm}
\usepackage{makecell}
\usepackage{arydshln}
\usepackage{supertabular}
\usepackage{booktabs}
\usepackage{color}

\newtheorem{theorem}{Theorem}
\newtheorem{lemma}[theorem]{Lemma}
\newtheorem{remark}[theorem]{Remark}
\newtheorem{proposition}[theorem]{Proposition}
\newtheorem{corollary}[theorem]{Corollary}

\newtheorem{example}[theorem]{Example}

\usepackage{blindtext}

\ifCLASSINFOpdf

\else

\fi

\begin{document}
\begin{center}
{\bf\large  Skew-product groups of finite abelian $p$-groups}
\end{center}

\begin{center}
{\small Jiawei He  \footnotemark}\\
\medskip
 {\small
  School of Mathematics and Information Science,\\ Nanchang Hangkong University, \\Nanchang, P.R.China
}
\end{center}

\footnotetext{Corresponding author: hjwywh@mails.ccnu.edu.cn (J. He).}

\renewcommand{\thefootnote}{\empty}

\begin{abstract}
A skew-morphism of a finite group $G$ is a permutation $\sigma$ of $G$ fixing the identity element such that there exists a function $\pi\colon G\to\mathbb{Z}$ satisfying $\sigma(xy)=\sigma(x)\sigma^{\pi(x)}(y)$ for all $x,y\in G$. For a given skew-morphism $\sigma$ of $G$, the product of the left regular representation of $G$ and the cyclic group $\langle\sigma\rangle$ forms a permutation group on $G$, called a skew-product group of $G$. Skew-product groups of finite elementary abelian $p$-groups were investigated in \cite{DLYZ}. In the present paper, we extend this study to all finite abelian $p$-groups.
\end{abstract}

\section{Introduction}
Throughout the paper all groups are assumed to be finite. A \emph{skew-morphism} of a group \(G\) is a permutation \(\sigma\) of \(G\) fixing the identity element, together with a function \(\pi\colon G\to \mathbb{Z}\) such that
\(
\sigma(gh) = \sigma(g)\sigma^{\pi(g)}(h) \quad\text{for all } g,h\in G.
\)
If \(\pi(g)=1\) for every \(g\in G\), then \(\sigma\) is an automorphism of \(G\); thus skew-morphisms naturally generalize group automorphisms.
Skew-morphisms are intimately connected with group factorizations. Let \(L_G = \{L_g \mid g\in G\}\) be the left regular representation of \(G\). For all \(g,h\in G\),
\(
(\sigma L_g)(h) = \sigma(g)\sigma^{\pi(g)}(h) = (L_{\sigma(g)}\sigma^{\pi(g)})(h),
\)
hence \(\langle\sigma\rangle L_G \subseteq L_G\langle\sigma\rangle\). Equality follows by comparing orders, so \(X := L_G\langle\sigma\rangle\) is a permutation group on \(G\), called the \emph{skew-product} of \(L_G\) by \(\sigma\). In \(X\), the cyclic subgroup \(\langle\sigma\rangle\) is a point stabilizer of the transitive group \(X\) and is therefore core-free. Conversely, every group factorization \(X = GY\) with \(Y = \langle y\rangle\) cyclic, \(G\cap Y = 1\) and \(Y\) core-free in \(X\) determines a skew-morphism of \(G\).

\vskip 3mm
Skew-morphisms were originally introduced to study regular \emph{Cayley maps}. A Cayley map \(\mathcal{M} = \operatorname{Cay}(G,S,P)\) is an embedding of a Cayley graph of \(G\) into an orientable closed surface, where the local rotation at each vertex is given by a cyclic permutation \(P\) of the generating set \(S\). The map is \emph{regular} if its automorphism group acts regularly on the directed edges. It was proved in~\cite{JS} that \(\mathcal{M}\) is regular precisely when \(P\) extends to a skew-morphism of \(G\) that has a generating orbit closed under inverses. Hence classifying regular Cayley maps for a group \(G\) is equivalent to classifying such skew-morphisms.

\vskip 3mm
Determining skew-morphisms for specific families of groups is a fundamental but challenging problem. For cyclic groups, skew-morphisms related to regular Cayley maps were extracted in~\cite{CT,YSK}; coset-preserving skew-morphisms of cyclic groups were fully described in~\cite{BJ1,BJ2}. For abelian groups, the skew-product groups of finite elementary abelian \(p\)-groups were studied in~\cite{DLYZ}. In the present paper we extend this line of research to all finite abelian \(p\)-groups.
Our first main theorem establishes a uniform normal Sylow structure for arbitrary skew-product groups over abelian \(p\)-groups.

\begin{theorem}\label{thm:A}
Let \(G\) be a finite abelian \(p\)-group and let \(X=G\langle\sigma\rangle\) be the associated skew-product group. Write \(|\sigma|=kp^m\) with \((k,p)=1\). Then:
\begin{enumerate}
\item \(X'\) is an abelian \(p\)-group;
\item \(P=G\langle\sigma^k\rangle\) is the unique Sylow \(p\)-subgroup of \(X\);
\item \(X=P\rtimes\langle\sigma^{p^m}\rangle\), where \(\langle\sigma^{p^m}\rangle\) acts faithfully on both \(P\) and its Frattini quotient \(P/\Phi(P)\).
\end{enumerate}
\end{theorem}

An immediate consequence is that every skew-morphism of an abelian \(p\)-group whose order is coprime to \(p\) is an automorphism. This reduces the general theory to the study of skew-product \(p\)-groups.
Our second main theorem concerns the structure of skew-product \(p\)-groups and their core quotients. Let \(N=\operatorname{Core}_X(G)\) denote the largest normal subgroup of \(X\) contained in \(G\).

\begin{theorem}\label{thm:B}
Let \(X=G\langle\sigma\rangle\) be a skew-product \(p\)-group. Then:
\begin{enumerate}
\item \(Z(X)\le N\); in particular, \(N\ne 1\) whenever \(X\ne 1\);
\item if \(X\) has nilpotency class at most \(2\), then \(G\trianglelefteq X\) and \(\sigma\) is an automorphism;
\item if \(G\not\trianglelefteq X\), then \(X/N\) is a nonabelian metabelian \(p\)-group with nontrivial cyclic center, contained in the image of \(\langle\sigma\rangle\).
\end{enumerate}
\end{theorem}

Under the additional hypothesis that the image of \(\langle\sigma\rangle\) is normal in \(X/N\), we derive sharp bounds on the size of \(G/N\) and show that \(X/N\) is metacyclic for odd \(p\). We emphasize that this normality condition is not automatic. This demonstrates that regularity alone is insufficient to imply normality, and that the structure of core quotients of skew-product groups is more subtle than one might initially expect.

\vskip 3mm\vskip 3mm
The paper is organized as follows. Section \ref{sec:prelim} reviews the equivalence between skew-morphisms and exact factorizations and establishes basic properties. Section \ref{sec:sylow} proves Theorem \ref{thm:A} and its corollaries. Section \ref{sec:frattini} discusses the faithful action on the Frattini quotient. Section \ref{sec:p-groups} develops the basic theory of skew-product \(p\)-groups and proves Theorem \ref{thm:B}.

\section{Preliminaries: exact factorizations and skew-morphisms}\label{sec:prelim}

Throughout this paper, we identify \(G\) with its left regular image \(L_G\) inside \(\operatorname{Sym}(G)\), so we write \(X=G\langle\sigma\rangle\) for the skew-product group.

\begin{proposition}\label{prop:normal-auto}
In a skew-product group \(X=G\langle\sigma\rangle\), the following are equivalent:
\begin{enumerate}
\item \(G\trianglelefteq X\);
\item \(\sigma\) is a group automorphism of \(G\).
\end{enumerate}
\end{proposition}

\begin{proof}
2)\(\Rightarrow\)1): If \(\sigma\) is an automorphism, then conjugation by \(\sigma\) sends \(L_g\) to \(L_{\sigma(g)}\):
\[
\sigma L_g\sigma^{-1}(h)=\sigma(g\sigma^{-1}(h))=\sigma(g)h=L_{\sigma(g)}(h).
\]
Thus \(\sigma\) normalizes \(G\), and since \(G\) generates itself, \(G\trianglelefteq X\).

1)\(\Rightarrow\)2): Suppose \(G\trianglelefteq X\). Then conjugation by \(\sigma\) restricts to an automorphism of \(G\). Since \(\sigma(1)=1\), we have
\[
L_{\sigma(gh)}=\sigma L_{gh}\sigma^{-1}
=(\sigma L_g\sigma^{-1})(\sigma L_h\sigma^{-1})
=L_{\sigma(g)}L_{\sigma(h)}
=L_{\sigma(g)\sigma(h)}.
\]
The left regular representation is faithful, so \(\sigma(gh)=\sigma(g)\sigma(h)\). Hence \(\sigma\) is a group automorphism.
\end{proof}

We shall use the following classical theorem of  \cite{Ito1955}.

\begin{theorem}\label{thm:Ito}
If a finite group can be written as a product of two abelian subgroups, then it is metabelian; that is, its commutator subgroup is abelian.
\end{theorem}

Finally, we record a standard fact about centralizers of regular abelian permutation groups.

\begin{lemma}\label{lem:centralizer-regular}
Let \(X\le\operatorname{Sym}(\Omega)\) be a faithful permutation group containing an abelian regular subgroup \(G\). Then
\[
Z(X)\le G.
\]
Consequently, \(Z(X)\le\operatorname{Core}_X(G)\).
\end{lemma}

\begin{proof}
It is a standard result in permutation group theory that the centralizer of a regular subgroup in the full symmetric group is the corresponding right regular representation. Since \(G\) is abelian, its left and right regular representations coincide. Therefore
\[
C_{\operatorname{Sym}(\Omega)}(G)=G.
\]
Every element of \(Z(X)\) centralizes \(G\), so \(Z(X)\le G\). Since \(Z(X)\trianglelefteq X\), it is contained in the largest normal subgroup of \(X\) inside \(G\), which is \(\operatorname{Core}_X(G)\).
\end{proof}

\section{Main results}

\subsection{The normal Sylow \(p\)-subgroup}\label{sec:sylow}

In this section we prove Theorem \ref{thm:A}, establishing that every skew-product group over an abelian \(p\)-group has a unique normal Sylow \(p\)-subgroup. Let \(X=GC\) be an exact factorization, where \(G\) is an abelian \(p\)-group and \(C\) is cyclic and core-free. Write \(|C|=kp^m\) with \((k,p)=1\), and let \(C_p\), \(C_{p'}\) denote the subgroups of \(C\) of orders \(p^m\) and \(k\), respectively. 
\begin{proof}
\textbf{Step 1: \(X'\) is abelian.} Since \(G\) and \(C\) are both abelian and \(X=GC\), Theorem \ref{thm:Ito} implies immediately that \(X'\) is abelian.

\textbf{Step 2: The \(p'\)-part of \(X'\) is trivial.} Let \(X'_{p'}\) be the subgroup of \(X'\) consisting of all elements of order coprime to \(p\). Since \(X'\) is abelian, \(X'_{p'}\) is characteristic in \(X'\). As \(X'\trianglelefteq X\), we have \(X'_{p'}\trianglelefteq X\).

Note that \(|X|=|G||C|\), and \(|G|\) is a power of \(p\), so the \(p'\)-part of \(|X|\) equals \(k=|C_{p'}|\). Thus \(C_{p'}\) is a Hall \(p'\)-subgroup of \(X\). Since \(X'_{p'}\trianglelefteq X\), the product \(X'_{p'}C_{p'}\) is a subgroup of \(X\). It is a \(p'\)-group, being a product of two normal \(p'\)-subgroups. By maximality of Hall \(p'\)-subgroups,
\[
X'_{p'}C_{p'}=C_{p'},
\]
which implies \(X'_{p'}\le C_{p'}\le C\). But \(X'_{p'}\trianglelefteq X\) and \(C\) is core-free, so \(X'_{p'}=1\). Therefore \(X'\) is a \(p\)-group. This proves (i).

\textbf{Step 3: Existence of a normal Sylow \(p\)-subgroup.} The quotient \(X/X'\) is abelian, so it has a unique Sylow \(p\)-subgroup \(Q\). Let \(P\) be the preimage of \(Q\) in \(X\). Then
\[
|P|=|X'|\cdot |Q|
\]
is a power of \(p\), since both \(X'\) and \(Q\) are \(p\)-groups. Moreover,
\[
|X:P|=|X/X':Q|
\]
is coprime to \(p\). Hence \(P\) is a Sylow \(p\)-subgroup of \(X\). Since \(Q\trianglelefteq X/X'\), we have \(P\trianglelefteq X\). By Sylow's theorems, \(P\) is the unique Sylow \(p\)-subgroup. This proves (ii).

\textbf{Step 4: Identification of \(P\).} Since \(P\) is the unique Sylow \(p\)-subgroup, it contains every \(p\)-subgroup of \(X\). In particular, \(G\subseteq P\) and \(C_p\subseteq P\), so
\[
GC_p\subseteq P.
\]
On the other hand, since \(G\cap C_p=1\), we have
\[
|GC_p|=|G|\cdot |C_p|=|G|\cdot p^m=|X|_p=|P|.
\]
Therefore \(GC_p=P\). In particular, the set product \(GC_p\) is a subgroup. 

\textbf{Step 5: Semidirect product decomposition.} Since \(P\trianglelefteq X\), \(P\cap C_{p'}=1\) (a \(p\)-group and a \(p'\)-group intersect trivially), and
\[
|P|\cdot |C_{p'}|=|X|_p\cdot |X|_{p'}=|X|,
\]
we conclude that \(X=P\rtimes C_{p'}\). This proves (iii).
\end{proof}

\begin{corollary}\label{cor:coprime-auto}
If the order of a skew-morphism of an abelian \(p\)-group is coprime to \(p\), then it is an automorphism.
\end{corollary}

\begin{proof}
If \((|\sigma|,p)=1\), then \(C_p=1\), so \(G=P\trianglelefteq X\). By Proposition \ref{prop:normal-auto}, \(\sigma\) is an automorphism.
\end{proof}

\begin{corollary}\label{cor:reduction}
Let \(\sigma\) be a skew-morphism of \(G\) with \(|\sigma|=kp^m\), \((k,p)=1\). Then \(\sigma^k\) is also a skew-morphism of \(G\), and its associated skew-product group is \(P=G\langle\sigma^k\rangle\).
\end{corollary}

\begin{proof}
The subgroup \(P\) acts faithfully on \(G\) (as a subgroup of \(\operatorname{Sym}(G)\)), contains the regular subgroup \(G\), and has point stabilizer \(P\cap\langle\sigma\rangle=\langle\sigma^k\rangle\). Hence this action defines a skew-morphism of \(G\), which is precisely \(\sigma^k\).
\end{proof}

\subsection{Faithful action on the Frattini quotient}\label{sec:frattini}

We now turn to the action of the \(p'\)-part of the cyclic stabilizer on the normal Sylow \(p\)-subgroup.

\begin{proposition}\label{prop:faithful-P}
Let \(X=GC\) be an exact factorization, where \(G\) is an abelian \(p\)-group and \(C\) is cyclic and core-free. Write \(|C|=kp^m\) with \((k,p)=1\), and let \(C_p\), \(C_{p'}\) denote the subgroups of \(C\) of orders \(p^m\) and \(k\), respectively. Then the conjugation action of \(C_{p'}\) on \(P\) is faithful.
\end{proposition}

\begin{proof}
Suppose \(c\in C_{p'}\) centralizes \(P\). Since \(C_{p'}\) is abelian, \(c\) also centralizes \(C_{p'}\). From the decomposition \(X=PC_{p'}\), it follows that \(c\) centralizes all of \(X\), i.e., \(c\in Z(X)\).
By Lemma \ref{lem:centralizer-regular}, \(c\in G\). But \(G\cap C_{p'}\le G\cap C=1\), so \(c=1\). Hence the action is faithful.
\end{proof}

A stronger result holds: the action remains faithful when passing to the Frattini quotient. This follows from the Burnside basis theorem.

\begin{theorem}\label{thm:frattini-faithful}
Let \(X=GC\) be an exact factorization, where \(G\) is an abelian \(p\)-group and \(C\) is cyclic and core-free. Write \(|C|=kp^m\) with \((k,p)=1\), and let \(C_p\), \(C_{p'}\) denote the subgroups of \(C\) of orders \(p^m\) and \(k\), respectively. Then the action of \(C_{p'}\) on \(P/\Phi(P)\) is faithful. Consequently, if \(d=\dim_{\mathbb F_p} P/\Phi(P)\) is the Frattini rank of \(P\), then
\[
C_{p'}\hookrightarrow \operatorname{GL}(d,p),
\]
and in particular
\[
k \mid \prod_{i=1}^d (p^i-1).
\]
\end{theorem}

\begin{proof}
By Proposition \ref{prop:faithful-P}, the natural map \(C_{p'}\to\operatorname{Aut}(P)\) is injective. By the Burnside basis theorem, the kernel of the restriction homomorphism
\[
\rho:\operatorname{Aut}(P)\longrightarrow \operatorname{Aut}(P/\Phi(P))
\]
is a \(p\)-group. Let \(K=\ker\rho\cap C_{p'}\). Then \(K\) is both a \(p\)-group and a \(p'\)-group, so \(K=1\). Therefore \(\rho|_{C_{p'}}\) is injective, meaning \(C_{p'}\) acts faithfully on \(P/\Phi(P)\).

The order formula for the general linear group is
\[
|\operatorname{GL}(d,p)|=p^{d(d-1)/2}\prod_{i=1}^d (p^i-1).
\]
Since \(C_{p'}\) is a \(p'\)-group, its order divides the \(p'\)-part of \(|\operatorname{GL}(d,p)|\), which is \(\prod_{i=1}^d (p^i-1)\).
\end{proof}

\begin{remark}
This theorem provides a strong number-theoretic restriction on the size of the \(p'\)-part of the stabilizer, depending only on the Frattini rank of \(P\). Since \(P=G\langle\sigma^k\rangle\), the Frattini rank of \(P\) is at most \(\operatorname{rank}(G)+1\), giving an explicit upper bound on \(k\).
\end{remark}

\subsection{Skew-product \(p\)-groups and proof of Theorem B}\label{sec:p-groups}

By Corollary \ref{cor:reduction}, the essential structural complexity of skew-product groups over abelian \(p\)-groups lies in the case where the stabilizer has \(p\)-power order. We now investigate this case in detail.
Throughout this section, assume \(X=G\langle\sigma\rangle\) is a skew-product \(p\)-group, i.e., \(|\sigma|=p^m\) for some \(m\ge 1\). Let
$
N=\operatorname{Core}_X(G).
$

\begin{proposition}\label{prop:core-nontrivial}
If \(X\ne 1\), then \(N\ne 1\).
\end{proposition}

\begin{proof}
A nontrivial finite \(p\)-group has nontrivial center. By Lemma \ref{lem:centralizer-regular}, \(1\ne Z(X)\le N\). Hence \(N\ne 1\).
\end{proof}

The following general equivalence holds in all skew-product groups, not just \(p\)-groups.

\begin{proposition}\label{prop:Xprime-G}
In a skew-product group \(X=G\langle\sigma\rangle\),
\[
G\trianglelefteq X \quad\Longleftrightarrow\quad X'\le G.
\]
\end{proposition}

\begin{proof}
If \(G\trianglelefteq X\), then \(X/G\) is cyclic, generated by the image of \(\sigma\). The commutator subgroup of any cyclic-by-abelian group is contained in the normal subgroup, so \(X'\le G\).

Conversely, if \(X'\le G\), then \(G/X'\) is a subgroup of the abelian group \(X/X'\). All subgroups of an abelian group are normal, so \(G/X'\trianglelefteq X/X'\). By the correspondence theorem, \(G\trianglelefteq X\).
\end{proof}

\subsection{Proof of Theorem B}

\begin{proof}\label{thm:B-proof}
(i) By Lemma \ref{lem:centralizer-regular}, \(Z(X)\le G\). Since \(Z(X)\trianglelefteq X\), it follows that \(Z(X)\le\operatorname{Core}_X(G)=N\). If \(X\ne 1\), then \(X\) is a nontrivial \(p\)-group, so \(Z(X)\ne 1\), and hence \(N\ne 1\).

(ii) If \(X\) has nilpotency class at most \(2\), then \(X'\le Z(X)\). By (i), \(Z(X)\le G\), so \(X'\le G\). Proposition \ref{prop:Xprime-G} gives \(G\trianglelefteq X\). By Proposition \ref{prop:normal-auto}, \(\sigma\) is an automorphism.

(iii) Assume \(G\not\trianglelefteq X\). Then \(N<G\). Set
\[
Y=X/N,\qquad A=G/N,\qquad B=\langle\sigma\rangle N/N.
\]
Since \(N\le G\) and \(G\cap\langle\sigma\rangle=1\), we have \(N\cap\langle\sigma\rangle=1\), so \(B\cong\langle\sigma\rangle\). Moreover, \(Y=AB\) is an exact factorization, \(A\) is abelian, \(B\) is cyclic, and
\[
\operatorname{Core}_Y(A)=\operatorname{Core}_X(G)/N=1.
\]
Thus \(A\) is core-free in \(Y\).

Since \(X\) is a \(p\)-group, so is \(Y\). If \(Y\) were abelian, then every subgroup of \(Y\) would be normal, contradicting \(\operatorname{Core}_Y(A)=1\) and \(A\ne 1\). Hence \(Y\) is nonabelian. By Itô's theorem, \(X\) is metabelian, and therefore its quotient \(Y\) is metabelian.

 Consider the action of \(Y\) on the left coset space \(Y/A\) by left multiplication. Since \(\operatorname{Core}_Y(A)=1\), this action is faithful. The subgroup \(B\) acts regularly on \(Y/A\): indeed, \(Y=AB\) implies transitivity, and \(|B|=|Y:A|\) implies regularity. Since \(B\) is cyclic and hence abelian, its centralizer in the full symmetric group on \(Y/A\) is \(B\) itself. Therefore every element of \(Z(Y)\) lies in \(B\). Thus
\[
1\ne Z(Y)\le B,
\]
where \(Z(Y)\ne 1\) because \(Y\) is a nontrivial \(p\)-group. As a subgroup of the cyclic group \(B\), \(Z(Y)\) is cyclic. Finally, \(Z(Y)\) is contained in the image of \(\langle\sigma\rangle\) in \(Y\).
\end{proof}

\begin{corollary}\label{cor:prime-order}
If \(|\sigma|=p\), then \(\sigma\) is an automorphism.
\end{corollary}

\begin{proof}
Since \(X\) is a \(p\)-group and \(|X:G|=|\sigma|=p\), the subgroup \(G\) has index \(p\). In any finite \(p\)-group, every subgroup of index \(p\) is normal. Hence \(G\trianglelefteq X\), and \(\sigma\) is an automorphism by Proposition \ref{prop:normal-auto}.
\end{proof}

\subsection{The quotient factorization}

Let
\[
Y=X/N,\qquad A=G/N,\qquad B=\langle\sigma\rangle N/N.
\]

\begin{proposition}\label{prop:quotient-factorization}
The quotient \(Y\) satisfies:
\begin{enumerate}
\item \(Y=AB\) is an exact factorization;
\item \(A\) is abelian, \(B\) is cyclic, and \(\operatorname{Core}_Y(A)=1\);
\item \(B\cong\langle\sigma\rangle\).
\end{enumerate}
\end{proposition}

\begin{proof}
Since \(N\le G\) and \(G\cap\langle\sigma\rangle=1\), we have \(N\cap\langle\sigma\rangle=1\). Therefore the projection \(\langle\sigma\rangle\to B\) is an isomorphism, proving (iii).

To see that \(A\cap B=1\), suppose \(gN=\sigma^i N\) for some \(g\in G\). Then \(g^{-1}\sigma^i\in N\le G\), so \(\sigma^i\in G\). But \(G\cap\langle\sigma\rangle=1\), so \(\sigma^i=1\) and \(gN=N\). Hence \(A\cap B=1\). The factorization \(Y=AB\) follows directly from \(X=G\langle\sigma\rangle\).

Finally, \(A=G/N\) is abelian as a quotient of an abelian group, and \(B\) is cyclic as a quotient of a cyclic group. By the correspondence theorem for normal subgroups,
\[
\operatorname{Core}_{X/N}(G/N)=\operatorname{Core}_X(G)/N=N/N=1.
\]
This completes the proof.
\end{proof}

\subsection{The center of the quotient lies in the cyclic factor}

Although \(\langle\sigma\rangle\) is core-free in \(X\), its image \(B\) need not be core-free in \(Y\). In fact, the opposite is true: the center of \(Y\) is always a nontrivial subgroup of \(B\). This is a dual version of Lemma \ref{lem:centralizer-regular}, with the roles of the regular subgroup and the stabilizer exchanged.

\begin{theorem}\label{thm:center-quotient}
Suppose \(Y\ne 1\). Then
\[
1\ne Z(Y)\le B.
\]
In particular, \(Z(Y)\) is cyclic and \(\operatorname{Core}_Y(B)\ne 1\). If \(G\not\trianglelefteq X\), then \(Y\) is nonabelian.
\end{theorem}

\begin{proof}
Consider the action of \(Y\) on the left coset space \(Y/A\) by left multiplication. Since \(\operatorname{Core}_Y(A)=1\), this action is faithful.

We claim that \(B\) acts regularly on \(Y/A\). Indeed, since \(Y=AB\), every coset \(yA\) can be written as \(bA\) for some \(b\in B\), so the action is transitive. The order of \(B\) equals \(|Y:A|\), so transitivity implies regularity.

Since \(B\) is cyclic and hence abelian, its centralizer in the full symmetric group on \(Y/A\) is \(B\) itself. Therefore
\[
C_{\operatorname{Sym}(Y/A)}(B)=B.
\]
Every element of \(Z(Y)\) centralizes \(B\), so under the faithful action on \(Y/A\), \(Z(Y)\) maps into \(C_{\operatorname{Sym}(Y/A)}(B)=B\). Hence \(Z(Y)\le B\).

Since \(Y\) is a nontrivial \(p\)-group, \(Z(Y)\ne 1\). As a subgroup of the cyclic group \(B\), \(Z(Y)\) is cyclic. Since \(Z(Y)\trianglelefteq Y\), we have \(1\ne Z(Y)\le\operatorname{Core}_Y(B)\).

Now suppose \(G\not\trianglelefteq X\). Then \(N<G\), so \(A\ne 1\). If \(Y\) were abelian, then \(A\trianglelefteq Y\), contradicting \(\operatorname{Core}_Y(A)=1\). Hence \(Y\) is nonabelian.
\end{proof}

\subsection{The case where the cyclic factor is normal}

Theorem \ref{thm:center-quotient} shows that the center of \(Y\) lies in \(B\), but it does not say that \(B\) itself is normal. We now analyze what happens under the stronger assumption that \(B\trianglelefteq Y\).

\begin{theorem}\label{thm:B-normal}
Let \(Y=AB\) be an exact factorization with \(A\) abelian, \(B\) cyclic, and \(\operatorname{Core}_Y(A)=1\). If \(B\trianglelefteq Y\), then
\[
Y=B\rtimes A,
\]
and the conjugation action of \(A\) on \(B\) is faithful.
\end{theorem}

\begin{proof}
Since \(B\trianglelefteq Y\) and \(A\cap B=1\), we immediately have the semidirect product decomposition \(Y=B\rtimes A\).

Let \(K=C_A(B)\) be the kernel of the action. Since \(A\) is abelian, \(K\trianglelefteq A\). Since \(B\) centralizes \(K\) by definition, \(B\) normalizes \(K\). Therefore
\[
K\trianglelefteq AB=Y.
\]
But \(K\le A\) and \(\operatorname{Core}_Y(A)=1\), so \(K=1\). Hence the action is faithful.
\end{proof}

\begin{corollary}\label{cor:B-normal-bounds}
In the notation of Proposition \ref{prop:quotient-factorization}, suppose that \(\langle\sigma\rangle N\trianglelefteq X\). Then there is an embedding
\[
G/N\hookrightarrow \operatorname{Aut}(\langle\sigma\rangle).
\]
If \(G\not\trianglelefteq X\) and \(|\sigma|=p^m\), then:
\begin{enumerate}
\item for odd \(p\), \(G/N\) is cyclic and \(1<|G:N|\mid p^{m-1}\); moreover, \(X/N\) is a nonabelian metacyclic \(p\)-group;
\item for \(p=2\) and \(m\ge 3\), \(G/N\) has rank at most \(2\), and its exponent divides \(2^{m-2}\).
\end{enumerate}
\end{corollary}

\begin{proof}
The hypothesis \(\langle\sigma\rangle N\trianglelefteq X\) is equivalent to \(B\trianglelefteq Y\). By Theorem \ref{thm:B-normal}, \(A=G/N\) embeds faithfully into \(\operatorname{Aut}(B)\cong\operatorname{Aut}(\langle\sigma\rangle)\).

For odd \(p\), \(\operatorname{Aut}(C_{p^m})\) is cyclic of order \(p^{m-1}(p-1)\). Since \(G/N\) is a \(p\)-group, its order divides \(p^{m-1}\). As a subgroup of a cyclic group, \(G/N\) is cyclic. Since \(G\not\trianglelefteq X\), we have \(|G:N|>1\). The quotient \(Y=B\rtimes A\) is metacyclic as an extension of a cyclic group by a cyclic group, and nonabelian by Theorem \ref{thm:center-quotient}.

For \(p=2\) and \(m\ge 3\), \(\operatorname{Aut}(C_{2^m})\cong C_2\times C_{2^{m-2}}\), which has rank \(2\) and exponent \(2^{m-2}\). Any abelian \(2\)-group embedding into it must have rank at most \(2\) and exponent dividing \(2^{m-2}\).
\end{proof}

In a transitive permutation group, a regular cyclic subgroup is not necessarily normal. Its normalizer is its holomorph \(B\rtimes\operatorname{Aut}(B)\), but a transitive group containing \(B\) may extend far beyond this normalizer. The following example demonstrates this phenomenon even for \(p\)-groups with abelian point stabilizers.

\begin{example}\label{ex:wreath}
Let \(p=3\) and let
\[
Y=C_3\wr C_3=(\langle t_0\rangle\times\langle t_1\rangle\times\langle t_2\rangle)\rtimes\langle s\rangle
\]
be the standard wreath product, where \(s^3=1\) and \(s t_i s^{-1}=t_{i+1}\). Then \(|Y|=3^4=81\), and \(Y\) has nilpotency class \(3\).
Consider the natural action of \(Y\) on \(\Omega=\mathbb F_3\times\mathbb F_3\) defined by:
\[
t_i(k,l)=(k,l+\delta_{ik}),\qquad s(k,l)=(k+1,l),
\]
where \(i,k,l\in\mathbb F_3\). This action is faithful and transitive. The stabilizer of the point \((0,0)\) is
\[
A=\langle t_1,t_2\rangle\cong C_3\times C_3,
\]
which has order \(9\). Since the action is faithful, \(\operatorname{Core}_Y(A)=1\).

Now define
\[
b=t_0s,\qquad B=\langle b\rangle.
\]
We compute the order of \(b\):
\[
b^3=t_0s\cdot t_0s\cdot t_0s
=t_0(st_0s^{-1})(s^2t_0s^{-2})s^3
=t_0t_1t_2.
\]
The element \(t_0t_1t_2\) has order \(3\), so \(b^9=1\) and \(b^3\ne 1\). Hence \(|b|=9\) and \(B\cong C_9\).

We verify that \(Y=AB\) is an exact factorization.
 The unique subgroup of \(B\) of order \(3\) is \(\langle b^3\rangle=\langle t_0t_1t_2\rangle\). But
\[
t_0t_1t_2(0,0)=(0,1)\ne(0,0),
\]
so \(t_0t_1t_2\notin A\). Therefore no non-trivial power of \(b\) lies in \(A\), so \(A\cap B=1\).
However, \(B\) is \emph{not} normal in \(Y\). Suppose for contradiction that \(B\trianglelefteq Y\). Then by Theorem \ref{thm:B-normal}, \(A\) would embed into \(\operatorname{Aut}(B)=\operatorname{Aut}(C_9)\). But \(|\operatorname{Aut}(C_9)|=\phi(9)=6\), while \(|A|=9\), and \(9\) does not divide \(6\). This is a contradiction. Hence \(B\not\trianglelefteq Y\).
\end{example}

We now present concrete examples illustrating the main theorems.

\subsection{Coprime-order automorphism: the alternating group \(A_4\)}

Let \(G=C_2\times C_2\) be the Klein four-group, and let \(\alpha\) be an automorphism of \(G\) of order \(3\). Then
\[
X=G\rtimes\langle\alpha\rangle\cong A_4.
\]
This is a skew-product group with \(\sigma=\alpha\). Here \(p=2\), \(|\sigma|=3\) is coprime to \(p\), so by Corollary \ref{cor:coprime-auto}, \(\sigma\) is an automorphism, consistent with the construction.
Clearly, the normal Sylow \(2\)-subgroup is \(P=G\), \(X=P\rtimes\langle\sigma\rangle\) and the Frattini quotient \(P/\Phi(P)\cong C_2\times C_2\) has dimension \(d=2\),
verifying Theorem \ref{thm:frattini-faithful}.

\subsection{Order-\(p\) automorphism: the dihedral group \(D_8\)}

Let \(G=\langle a,b\rangle\cong C_2\times C_2\), and let \(\alpha\) be the involution swapping \(a\) and \(b\). Then
\[
X=G\rtimes\langle\alpha\rangle\cong D_8,
\]
the dihedral group of order \(8\). This is a skew-product \(2\)-group with \(|\sigma|=2=p\). By Corollary \ref{cor:prime-order}, \(\sigma\) must be an automorphism. Note that \(X\) has nilpotency class \(2\), so \(G\) is normal and \(\sigma\) is an automorphism, again consistent.

\subsection{The Heisenberg group: class-\(2\) automorphism}

Let \(p\) be an odd prime and let \(X\) be the Heisenberg group of order \(p^3\):
\[
X=\langle x,y,z\mid x^p=y^p=z^p=1,\ [x,y]=z,\ z\in Z(X)\rangle.
\]
Set \(G=\langle x,z\rangle\cong C_p\times C_p\) and \(C=\langle y\rangle\cong C_p\). Then \(X=GC\) is an exact factorization.
Since \(|X:G|=p\), \(G\trianglelefteq X\), so the corresponding skew-morphism is an automorphism (conjugation by \(y\)). The cyclic subgroup \(C\) is core-free: it has prime order and is not normal (\(xyx^{-1}=yz\notin C\)), so its core is trivial.

\subsection{Mixed order: combining \(p\) and \(p'\) parts}

Let \(G=V\oplus W\), where \(V\cong W\cong C_2\times C_2\). Choose \(\alpha\in\operatorname{Aut}(V)\) of order \(3\) and \(\beta\in\operatorname{Aut}(W)\) of order \(2\). Define
\[
\sigma=\alpha\oplus\beta\in\operatorname{Aut}(G).
\]
Then \(|\sigma|=6\), and \(X=G\rtimes\langle\sigma\rangle\) is a skew-product group.
Here \(k=3\), \(m=1\), \(p=2\). The normal Sylow \(2\)-subgroup is
\[
P=G\langle\sigma^3\rangle=G\rtimes\langle\beta\rangle,
\]
and \(X=P\rtimes\langle\sigma^2\rangle\).
The subgroup \(\langle\sigma^2\rangle\) has order \(3\) and acts faithfully on \(P/\Phi(P)\). The Frattini rank of \(P\) is \(3\) or \(4\), and
\[
3\mid \prod_{i=1}^d(2^i-1)
\]
holds for \(d\ge 2\), consistent with Theorem \ref{thm:frattini-faithful}.

\end{document}